\documentclass[12pt, reqno]{amsart}
\usepackage{amsmath, amsthm, amscd, amsfonts, amssymb, graphicx, color, txfonts}
\usepackage[bookmarksnumbered, colorlinks, plainpages]{hyperref}

\newcommand{\Fock}{\mathcal{F}^{2}(\mathbb{C})}
\newcommand{\rad}{\mathrm{rad}}
\newcommand{\Rad}{\operatorname{Rad}}
\newcommand{\Complex}{\mathbb{C}}
\newcommand{\innerp}[2]{\left\langle#1,#2\right\rangle}
\newcommand{\BF}{\mathcal{B}(\mathcal{F}^{2}(\mathbb{C}))}
\newcommand{\Pn}{\mathrm{P}}

\newcommand{\g}{\mathrm{g}}

  \newcommand{\B}{\mathcal{B}}
 \newcommand{\M}{\mathcal{M}}

 \newcommand{\N}{\mathbb{N}}
 
 \newcommand{\X}{\mathcal{X}}
  \newcommand{\Disk}{\mathbb{D}}

 \newcommand{\Real}{\mathbb{R}}
 \newcommand{\Entero}{\mathbb{Z}}

\newcommand{\AlgToeplitz}{\mathcal{T}_{\mathrm{rad}}}
 \newcommand{\RPlus}{\Real_{+}}
 
 \newcommand{\Ele}{L_{2}}
 \newcommand{\Bergman}{\mathcal{A}^{2}}

\newcommand{\SQRTO}{\operatorname{RO}(\mathbb{Z}_+)} 
\newcommand{\Gammas}{\mathfrak{G}}

\newtheorem{theorem}{Theorem}[section]

\newtheorem{proposition}[theorem]{Proposition}
\newtheorem{lemma}[theorem]{Lemma}
\theoremstyle{definition}
\newtheorem{definition}[theorem]{Definition}
\newtheorem{example}[theorem]{Example}
\newcommand{\eps}{\varepsilon}

\usepackage[colorinlistoftodos]{todonotes}

\definecolor{darkgreen}{rgb}{0,0.7,0}

\numberwithin{equation}{section}
\begin{document}
\setcounter{page}{1}

\title[Radial Toeplitz operators on the Fock space]{Radial Toeplitz operators on the Fock space and square-root-slowly oscillating sequences}

\author[K. Esmeral, E. A.  Maximenko]{Kevin Esmeral$^1$,  Egor  A.  Maximenko$^{2}$ }

\address{$^{1}$ Department of Mathematics\\CINVESTAV-IPN \\M\'exico, D.F, 07360, Mexico}
\email{\textcolor[rgb]{0.00,0.00,0.84}{matematikoua@gmail.com}}

\address{$^{2}$ Escuela Superior de F\'isica y Matem\'aticas\\ Instituto Polit\'ecnico Nacional\\ M\'exico,  D.F, 07730, Mexico}
\email{\textcolor[rgb]{0.00,0.00,0.84}{maximenko@esfm.ipn.mx}}





\date{\today.\\
  $\,{}\,{}\,{}\,{}\;{}$      The work on the paper was partically supported
by  CONACyT and by IPN-SIP project 20150422.}

\begin{abstract}
In this paper we show that the C*-algebra generated by radial Toeplitz operators
with $L_{\infty}$-symbols acting on the Fock space is isometrically isomorphic to the C*-algebra of bounded sequences
uniformly continuous with respect to the square-root-metric
$\rho(j,k)=|\sqrt{\vphantom{jk}j}-\sqrt{\vphantom{jk}k}\,|$.
More precisely, we prove that the sequences of eigenvalues
of radial Toeplitz operators form a dense subset
of the latter C*-algebra of sequences.
\end{abstract}

\maketitle

\section{Introduction}

Let $\Fock$ be the Fock space 
(also known as the Segal--Bargmann space,
see \cite{Bargmann,BerezinFA,Fock,Segal})
consisting of all entire functions that are square integrable
with respect to the Gaussian measure
\[
d\g(z)=\frac{1}{\pi}\,e^{-|z|^{2}} d\nu(z),\quad z\in\Complex,
\]
where $\nu$ is the usual Lebesgue measure on $\Complex$. 


Given $\varphi\in\,L_{\infty}(\Complex)$,
the \emph{Toeplitz operator} $T_{\varphi}$ with defining symbol $\varphi$
acts on the Fock space $\Fock$  by the rule $T_{\varphi}f=\Pn(f\varphi)$, where $\Pn$ stays for the Bargmann projection from $\Ele(\Complex,d\g)$ onto $\Fock$. Operators of this kind have been extensively studied 
\cite{Bauer-Issa,Grudsky-Vasilevski,Isra-Zhu, Zhu-F},
particularly in connection
with quantum mechanics \cite{Berger-Coburn,Coburn},
harmonic analysis \cite{Abreu-Faustino,M-Englis}, etc.

The C*-algebra generated by Toeplitz operators with $L_{\infty}$-symbols is not commutative, however, there are classes of defining symbols that generate commutative C*-algebras of Toeplitz operators.
In particular, many authors \cite{BauerHerreraYanezVasilevski2014b, Bauer-Le,GrudskyKaraVasilevski, GrudskyMaximenkoVasilevski2013,Grudsky-Vasilevski,Zorboska} have studied the Toeplitz operators
with radial defining symbols
acting on various spaces of analytic or harmonic functions
over the unit ball or over the whole space $\mathbb{C}^n$. The main reason is that these operators are \emph{diagonal}
in the monomial basis,
which provides easy access to their properties. 

For radial symbols on the unit disk $\Disk$,
intensive investigation by Su\'{a}rez \cite{SuarezD}
complemented by Grudsky, Maximenko and Vasilevski
\cite{GrudskyMaximenkoVasilevski2013}
showed that the C*-algebra generated by radial Toeplitz operators
acting on the Bergman space $\Bergman(\Disk)$
is isometrically isomorphic to the C*-algebra
$\operatorname{SO}(\Entero_{+})$
consisting of the bounded sequences that \emph{slowly oscillate}
in the sense of R.~Schmidt:
\[
\operatorname{SO}(\Entero_{+})=\left\{\sigma\in\ell_{\infty}(\Entero_{+})\colon \lim_{\frac{j+1}{k+1}\rightarrow1}|\sigma_{j}-\sigma_{k}|=0\right\},
\]
where $\Entero_{+}=\N\cup\{0\}$.
In other words, $\operatorname{SO}(\Entero_{+})$
consists of bounded functions
$\sigma\colon\Entero_{+}\longrightarrow\Complex$ uniformly continuous with respect to the \emph{logarithmic metric}
\[
\eta(j,k)=|\ln(j+1)-\ln(k+1)|.
\]
This result extends to weighted Bergman spaces over the unit ball
\cite{BauerHerreraYanezVasilevski2014b,HerreraMaximenkoVasilevski}.
Recent studies \cite{EsmeralMaximenko,GrudskyMaximenkoVasilevski2013,HerreraMaximenkoHutnik,SuarezD} have given explicit descriptions of commutative C*-algebras generated by vertical and angular Toeplitz operators on Bergman spaces.



This paper focuses on studying the C*-algebra generated by radial Toeplitz operators acting on Fock spaces. It is well known \cite{Zhu-F} that the normalized monomials $e_{n}(z)=z^{n}/\sqrt{n!},$ $n\in\Entero_{+}$,
form an orthonormal basis of $\Fock$, and that the Toeplitz operators with bounded radial symbols are diagonal with respect to this basis \cite{Grudsky-Vasilevski}. Namely, if $a\in\,L_{\infty}(\Real_{+})$ and $\varphi(z)=a(|z|)$ a.e.  $z\in\Complex$,  then $T_{\varphi}e_{n}=\gamma_{a}(n)e_{n}$, where 
\begin{equation}\label{gamma-radial-fock}
\gamma_{a}(n)=\frac{1}{\sqrt{n!}}\int_{\Real_{+}}a(\sqrt{r}\,)e^{-r}r^{n}\,dr,\quad n\in\Entero_{+}. 
\end{equation}
From this diagonalization
we have that the C*-algebra $\AlgToeplitz$
generated by Toeplitz operators with radial $L_{\infty}$-symbols
is isometrically isomorphic to the C*-algebra $\mathcal{G}$
generated by the set $\mathfrak{G}$
of all sequences of the eigenvalues:
\begin{equation}
\mathfrak{G}=\left\{\gamma_{a}\colon
a\in\,L_{\infty}(\Real_{+})\right\}.
\end{equation}

%
%

The main result of this paper states that the uniform closure
of $\mathfrak{G}$ coincides with the C*-algebra
$\SQRTO$ consisting of all bounded sequences
$\sigma\colon\Entero_{+}\longrightarrow\Complex$
that are uniformly continuous
with respect to the \emph{sqrt-metric}
\[
\rho(m,n)=\left|\sqrt{m}-\sqrt{n}\ \right|,\quad m,n\in\Entero_{+}.
\]
As a consequence, we obtain an explicit description of the C*-algebra $\mathcal{G}$ generated by $\mathfrak{G}$:
\[
\mathcal{G}=\SQRTO.
\]
Surprisingly for us, the C*-algebra $\mathcal{G}$ turns out to be wider than the class of sequences $\operatorname{SO}(\Entero_{+})$
obtained for the radial case on weighted Bergman spaces.

The  results of the paper can be generalized to radial Toeplitz operators on the multi-dimensional Fock space $\mathcal{F}^{2}(\Complex^{n},\left(\alpha/\pi\right)^{n}e^{-\alpha|z|^{2}}dv_{n}(z))$; in this case  the eigenvalue  associated to the element $e_{\beta}$ of the canonical basis depends only on the length of the multi-index $\beta$,
as in \cite{GrudskyMaximenkoVasilevski2013}.

The paper is organized as follows.
In Section~\ref{radial-toeplitz-operators}
we have compiled some basic facts about radial Toeplitz operators in Fock space.
In Sections~\ref{Sqrt-oscillating-sequences}
and \ref{sqrt-property-of-gammas}
we introduce the class $\SQRTO$
and prove that $\mathfrak{G}$ is contained in $\SQRTO$.
%
%
%

The major part of the paper is occupied by
a proof  that $\mathfrak{G}$ is \emph{dense}
in $\SQRTO$,
see a scheme in Figure~\ref{fig:scheme}.
Given a sequence $\sigma\in\SQRTO$,
we extend it to a sqrt-oscillating function $f$ on $\Real_+$
(Proposition~\ref{prop:contuc}).
After the change of variables $h(x)=f(x^2)$
we obtain a bounded and uniformly continuous
function $h$ on $\Real$.
In Section \ref{approx-unit-cont-func-by-conv},
using Dirac sequences and Wiener's division lemma,
we show that functions from $C_{b,u}(\Real)$
can be uniformly approximated by convolutions $k\ast b$,
where $k$ is a fixed $L_1(\Real)$-function
whose Fourier transform does not vanish on $\Real$.
This construction will be applied when
$k$ is the heat kernel $H$. In Section~\ref{approx-convoluzation} we examine the asymptotic behavior of the eigenvalues' sequences $\gamma_{a}$.
It is shown there that after change of variables $\sqrt{n}=x$,
the function $x\mapsto\gamma_{a}(x^{2})$,
for $x$ sufficiently large,
is close to the convolution of the symbol $a$
with the heat kernel $H$.
In Section~\ref{density-in-c0} first we show that
$c_{0}(\Entero_{+})$ coincides with the uniform closure of the set
$\{\gamma_{a}\colon\ a\in L_{\infty}(\Real_{+}),\ 
\lim_{r\to\infty}a(r)=0\}$. After that, gathering together all the pieces, we obtain the main result. 
Finally, in Section \ref{Unbounded-symbols}
we describe a class of generating symbols
bigger than $L_\infty(\Real)$,
with eigenvalues' sequences still belonging to $\SQRTO$,
and construct an unbounded generating symbol $a$
such that $\gamma_a\in\ell_\infty(\Entero_+)\setminus\SQRTO$.
%
%
\begin{figure}[hb]
\includegraphics{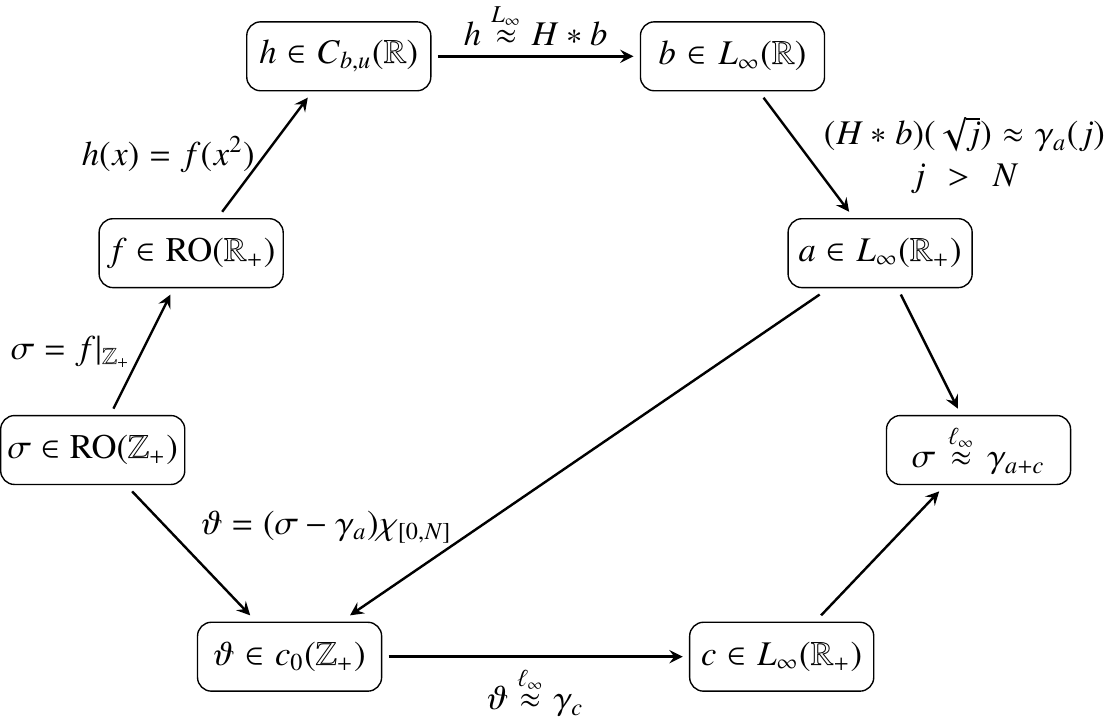}
\caption{Scheme of the proof of density: the upper chain represents
the approximation of $\sigma(j)$ for large values of $j$ ($j>N$),
and the lower one corresponds to the uniform approximation
of the sequence $\sigma-\gamma_a$
multiplied by the characteristic function $\chi_{[0,N]}$.}
\label{fig:scheme}
\end{figure}

\clearpage
\section{Radial Toeplitz operators}\label{radial-toeplitz-operators}

In this section we compile some basic facts on Toeplitz operators with radial symbols from $L_{\infty}(\Complex)$ acting on the Fock space $\Fock$.
Essentially, we repeat for the Fock space
the facts stated by Zorboska \cite{Zorboska}
for the Bergman space over the unit disk,
adding some ideas from \cite{Grudsky-Vasilevski}.

Let $t\in\Real$, and $U_{t}\colon\Fock\to\Fock$
be the unitary operator given by the composition of functions
with the rotation by the angle $t$ around the origin in the negative direction:
\begin{equation}
(U_{t}f)(z)=f(e^{-it}z),\quad z\in\Complex.
\end{equation}
For $S\in\BF$ we denote by $\Rad(S)$ the
\emph{radialization} of $S$ defined by
\begin{equation}
\Rad(S)=\frac{1}{2\pi}\int_{0}^{2\pi}U_{-t}SU_{t}\,dt,
\end{equation}
where the integral is understood in the weak sense.

\begin{definition}[radial operator acting on the Fock space]
Let $S\in\BF$. The operator $S$ is said to be \emph{radial} if it is invariant under rotations, that is,
if for every $t\in[0,2\pi)$,
\begin{equation}
SU_{t}=U_{t}S.
\end{equation}
\end{definition}
Observe that $S\in\BF$ is radial if and only if $\Rad(S)=S$.
\begin{definition}[radial function]\label{radial-function}
A function $\varphi\in\,L_{\infty}(\Complex)$ is called \emph{radial}
if there exists $a\in\,L_{\infty}(\Real_{+})$
such that $\varphi(z)=a(|z|)$ a.e. $z\in\Complex$.
\end{definition}

\begin{definition}[the radialization of a function]
Let $\varphi\in\,L_{\infty}(\mathbb{C})$. The function $\rad(\varphi)$ given by
\begin{equation}\label{rad:of:functions}
\rad(\varphi)(z)=\frac{1}{2\pi}\int_{0}^{2\pi}\varphi(e^{it}z)\,dt
\end{equation}
is called the \emph{radialization} of $\varphi$.
\end{definition}
By the periodicity of the mapping $t\mapsto e^{it}$,
the formula \eqref{rad:of:functions} can rewritten as
\begin{equation}\label{rad-f-modulo}
\rad(\varphi)(z)=\frac{1}{2\pi}\int_{0}^{2\pi}\varphi(e^{it}|z|)\,dt.
\end{equation}
\begin{lemma}[criterion for a function to be radial]
A function $\varphi\in\,L_{\infty}(\Complex)$ is radial if and only if $\varphi(z)=\rad(\varphi)(z)$ a.e.  $z\in\Complex$.
\end{lemma}
\begin{proof}
Suppose that $\varphi\in\,L_{\infty}(\Complex)$ is radial,
i.e. there exists $a\in\,L_{\infty}(\Real_{+})$
such that $\varphi(z)=a(|z|)$ a.e.  $z\in\Complex$.  Therefore,  by \eqref{rad-f-modulo} and by Fubini's theorem one gets that
\begin{align*}
\int_{\Complex}\rad(\varphi)(w)w^{n}\overline{w}^{m}d\g(w)
&=\frac{1}{2\pi}\int_{0}^{2\pi}\int_{\Real_{+}}\int_{0}^{2\pi}r^{n+m+1}e^{-r^{2}}\varphi(e^{i\alpha}r)\,e^{i\beta(n-m)}d\alpha\,drd\beta\\
&=\frac{1}{2\pi}\int_{0}^{2\pi}\int_{\Real_{+}}\int_{0}^{2\pi}r^{n+m+1}e^{-r^{2}}a\left(|e^{i\alpha}r|\right)e^{i\beta(n-m)}d\alpha\,drd\beta\\
&=\left(\int_{\Real_{+}}a(r)r^{n+m+1}e^{-r^{2}}dr\right)\left(\int_{0}^{2\pi}e^{i\beta(n-m)}d\beta\right)\\
&=\int_{\Complex}\varphi(w)w^{n}\overline{w}^{m}d\g(w),\quad n,m\in\Entero_{+}.
\end{align*}
Now, since  $\rad(\varphi)-\varphi$ belongs to $\Ele(\Complex,d\g)$ and the span of  $\{w^{m}\overline{w}^{n}\colon m,n\in\Entero_{+}\}$ is dense in $\Ele(\Complex,d\g)$, we obtain that $\rad(\varphi)(z)=\varphi(z)$ a.e. $z\in\Complex$. 

Conversely, if $\varphi(z)=\rad(\varphi)(z)$ a.e $z\in\Complex$,
then by \eqref{rad-f-modulo} one gets that $\varphi(z)=\rad(\varphi)(|z|)$ for a.e. $z\in\Complex$,
which means that the condition of Definition \ref{radial-function} holds with $a(r)=\rad(\varphi)(r)$.
\end{proof}

The \emph{Berezin transform} \cite{BerezinFA,Zorboska} plays an important role in the description of properties of bounded operators, in particular for Toeplitz operators. The Berezin transform of a bounded operator $S$ on the Fock space $\Fock$
is the function $\widetilde{S}$ defined by
\begin{equation}
\widetilde{S}(z)=\frac{\innerp{Sk_{z}}{k_{z}}}{\innerp{k_{z}}{k_{z}}},\quad z\in\Complex,
\end{equation}
where the function $k_{z}(w)=e^{\overline{z}w}$
is the reproducing kernel of $\Fock$.

The next result provides a criterion for an operator to be radial. It mimics a  result given by Zorboska~\cite{Zorboska} for operators on the Bergman space over the unit disk.

\begin{theorem}[criterion of radial operators]\label{criterion-radial}
Let $S\in\BF$. The following conditions are equivalent.
\begin{itemize}
\item [$(\mathrm{i})$] $S$ is radial.
\item [$(\mathrm{ii})$] $S$ is a diagonal operator with respect to the monomial basis.
\item [$(\mathrm{iii})$] The Berezin transform $\widetilde{S}$
is a radial function.
\end{itemize}
\end{theorem}


An easy computation shows that $\rad(\,\widetilde{\varphi}\,)=\widetilde{\rad(\varphi)}=\widetilde{T_{\rad(\varphi)}}$, for each $\varphi\in\,L_{\infty}(\Complex)$. Thus, by Theorem \ref{criterion-radial} and by injectivity of Berezin transform the following criterion holds.
\begin{proposition}\label{ToeplitzR-Fradial}
Let $\varphi\in\,L_{\infty}(\Complex)$. The Toeplitz operator $T_{\varphi}$ is  radial if and only if $\varphi$ is a  radial function.
\end{proposition}
%

\section{Sqrt-oscillating sequences}
\label{Sqrt-oscillating-sequences}

In this section we introduce formally the sets of sequences $\SQRTO$ and   functions $\operatorname{RO}([0,+\infty))$.
We also show that the sequences of the class $\SQRTO$
can be extended to functions of the class 
$\operatorname{RO}([0,+\infty))$.

Define $\rho\colon\mathbb{Z}_{+}\times\mathbb{Z}_{+}\to [0,+\infty)$ by
\begin{equation}\label{square-root-metric}
\rho(m,n)=\left|\sqrt{m}-\sqrt{n}\,\right|.
\end{equation}
The function $\rho$ is a metric on $\mathbb{Z}_{+}$
because it is obtained from the usual metric
\[
d\colon \mathbb{R}_{+}\times\mathbb{R}_{+}\to [0, +\infty),\quad
d(t, u) := |t-u|,
\]
via the injective function
$\mathbb{Z}_{+}\to\mathbb{R}_{+}$,
$m\mapsto\sqrt{m}$.

The   \emph{modulus of continuity}  with respect to the square-root-metric $\rho$ of a complex sequence $\sigma=(\sigma_{n})_{n\in\mathbb{Z}_{+}}$ is the function $\omega_{\rho,\sigma} \colon [0, +\infty)\to [0,+\infty]$ given by the rule
\begin{equation}
\omega_{\rho,\sigma}(\delta)
=\sup\left\{|\sigma_{n}-\sigma_{m}|\colon\ 
m,n\in\mathbb{Z}_{+},\ \rho(m,n)\leq\delta\right\}.
\end{equation}
We denote by $\SQRTO$ the set of the bounded sequences
that are uniformly continuous
with respect to the sqrt-metric:
\begin{equation}
\SQRTO =\left\{\sigma\in\ell_{\infty}(\mathbb{Z}_{+})\colon
\lim_{\delta\to 0}\omega_{\rho,\sigma}(\delta)=0\right\}.
\end{equation}

\begin{proposition}
$\SQRTO$ is a closed C*-subalgebra
of  $\ell_{\infty}(\mathbb{Z}_{+})$.
\end{proposition}
\begin{proof}
The proof of this fact runs as in \cite[Proposition 3.8]{GrudskyMaximenkoVasilevski2013}.
\end{proof}

The following simple criterion shows
that the Lipschitz-continuity of
\hyphenation{se-quen-ces}sequences
(with respect to the metric $\rho$)
can be described in terms of the differences between the adjacent elements.
\begin{proposition}\label{cond:Sup-Lip}
A sequence $\sigma\colon\colon\mathbb{Z}_{+}\to\Complex$ is Lipschitz continuous
with respect to $\rho$ if and only if
\begin{equation}
\sup_{n\in\mathbb{Z}_{+}}
\left(\sqrt{n+1}\,\left|\sigma(n+1)-\sigma(n)\right|\right)
<+\infty.
\end{equation}
\end{proposition}

\begin{proof}
Suppose that $\sigma$ is Lipschitz continuous with respect to $\rho$,
that is, there exists $M>0$ such that
$\left|\sigma(m)-\sigma(n)\right|\leq M\rho(m,n)$
for each $m,n\in\mathbb{Z}_{+}$.
Applying this inequality in the particular case
$m=n+1$, one gets
\begin{align*}
\sqrt{n+1}\,\left|\sigma(n+1)-\sigma(n)\right|
\leq M\left(\sqrt{n+1}-\sqrt{n}\,\right)\sqrt{n+1}
=\frac{M\sqrt{n+1}}{\sqrt{n+1}+\sqrt{n}}
\leq\,M.
\end{align*}
Conversely, suppose that
$\sup_{n}\left(\sqrt{n+1}\,\left|\sigma(n+1)-\sigma(n)\right|\right)=M<+\infty$.
Hence, if $n>m$, then we ``join'' $m$ with $n$
by the chain of the intermediate elements
and estimate the differences of the adjacent elements
using the hypothesis:
\begin{align*}
\left|\sigma(m)-\sigma(n)\right|&\leq\sum_{k=m}^{n-1}\left|\sigma(k+1)-\sigma(k)\right|=\sum_{k=m}^{n-1}\frac{\sqrt{k+1}+\sqrt{k}}{\sqrt{k+1}+\sqrt{k}}\left|\sigma(k+1)-\sigma(k)\right|\\
&\leq2\sum_{k=m}^{n-1}\frac{\sqrt{k+1}}{\sqrt{k+1}+\sqrt{k}}\left|\sigma(k+1)-\sigma(k)\right|\leq2M\sum_{k=m}^{n-1}\frac{1}{\sqrt{k+1}+\sqrt{k}}\\
&=2M\sum_{k=m}^{n-1}\left(\sqrt{k+1}-\sqrt{k}\right)= 2M\left(\sqrt{n}-\sqrt{m}\,\right)=2M\rho(n,m).	
\end{align*}
The same upper estimate can be drawn for $m\geq n$.
Thus, $\sigma$ is Lipschitz continuous with respect to $\rho$.
\end{proof}

The square root metric $\rho$ can be extended to the set $[0,+\infty)$:
\[
\rho(x,y)=\left|\sqrt{x}-\sqrt{y}\,\right|.
\]
Given a function $f\colon[0,+\infty)\to\mathbb{C}$,
its \emph{modulus of continuity} with respect to $\rho$ is defined by
\[
\omega_{\rho,f}(\delta)
=\sup\{|f(x)-f(y)|\colon\ x,y\in[0,+\infty),\ \rho(x,y)\le\delta\}.
\]
We denote by $\operatorname{RO}([0,+\infty))$ the C*-algebra of all bounded and uniformly continuous functions on $[0,+\infty)$ with respect to the  extended square root metric $\rho$:
\begin{equation}
\operatorname{RO}([0,+\infty))
=\left\{f\in\,C_{b,u}([0,+\infty))\colon\ 
\lim_{\delta\to 0}\omega_{\rho,f}(\delta)=0\right\}.
\end{equation}

If $f$ is a function of the class $\operatorname{RO}([0,+\infty))$,
then, obviously, its restriction to $\Entero_{+}$
is a sequence of the class $\SQRTO$.
We are going to show that \emph{every} sequence
of the class $\SQRTO$ can be obtained in this manner.
Our extension of sequences to functions
is just the piecewise-linear interpolation
with respect to the parameter $\tau(x)=\sqrt{x}$.

\begin{lemma} \label{lem:contestimate}
Let $\sigma\colon\Entero_{+}\to\mathbb{C}$.
Define $f\colon[0,+\infty)\to\mathbb{C}$ by
\begin{equation}\label{ext}
f(x) = \sigma(n) + \frac{\tau(x)-\tau(n)}{\tau(n+1)-\tau(n)} (\sigma(n+1)-\sigma(n)),
\end{equation}
where  $n=\lfloor x\rfloor$ and $\tau(x)=\sqrt{x}$.
Then $f|_{\Entero_{+}}=\sigma$, $\|f\|_\infty=\|\sigma\|_\infty$
and for every $\delta\in(0,1]$
\begin{equation}\label{omegabound}
\omega_{\rho,f}(\delta)
\le
3\max(\omega_{\rho,\sigma}(\sqrt{\delta}\,),\,\sqrt{\delta}\,\omega_{\rho,\sigma}(1)).
\end{equation}
\end{lemma}

\begin{proof}
Note that $f(x)$ in \eqref{ext} is a convex combination
of $\sigma(n)$ and $\sigma(n+1)$:
\begin{equation}\label{extconvex}
f(x) = \frac{\tau(n+1)-\tau(x)}{\tau(n+1)-\tau(n)}\,\sigma(n) + \frac{\tau(x)-\tau(n)}{\tau(n+1)-\tau(n)}\,\sigma(n+1).
\end{equation}
The first two assertions of the proposition are obvious.
Let us prove \eqref{omegabound}.
Fix $\delta\in(0,1]$ and suppose that $x,y\ge0$ with $\rho(x,y)\le\delta$.

Case I: $n\le x\le y\le n+1$ for some $n\in\Entero_{+}$.
In this case
\begin{align*}
|f(x)-f(y)|
=\frac{\tau(y)-\tau(x)}{\tau(n+1)-\tau(n)}\,|\sigma(n+1)-\sigma(n)|
\le \frac{\rho(x,y)\,\omega_{\rho,\sigma}(\rho(n,n+1))}{\rho(n,n+1)}.
\end{align*}
If $\rho(n,n+1)\le\sqrt{\delta}$, then
\[
|f(x)-f(y)| \le \frac{\rho(x,y)}{\rho(n,n+1)}\,\omega_{\rho,\sigma}(\sqrt{\delta})
\le \omega_{\rho,\sigma}(\sqrt{\delta}).
\]
If $\rho(n,n+1)\ge\sqrt{\delta}$, then
\[
|f(x)-f(y)|\le \frac{\delta}{\sqrt{\delta}}\,\omega_{\rho,\sigma}(1) = \sqrt{\delta}\,\omega_{\rho,\sigma}(1).
\]
In both subcases,
\begin{equation}\label{closeestimate}
|f(x)-f(y)|\le \max(\omega_{\rho,\sigma}(\sqrt{\delta}),\,\sqrt{\delta}\,\omega_{\rho,\sigma}(1)).
\end{equation}

Case II: $\lfloor x\rfloor = n < m = \lfloor y\rfloor$.
Then $\rho(n+1,m)\le\rho(x,y)\le\delta$, and
\[
|f(x)-f(y)|\le |f(x)-f(n+1)| + |f(n+1)-f(m)| + |f(m)-f(y)|.
\]
Applying the inequality $\rho(n+1,m)\le\rho(x,y)\le\delta$
and the result of Case I, we obtain
\begin{equation}\label{farestimate}
|f(x)-f(y)|\le \omega_{\rho,\sigma}(\delta) + 2\max(\omega_{\rho,\sigma}(\sqrt{\delta}),\,\sqrt{\delta}\,\omega_{\rho,\sigma}(1)).
\end{equation}
In both cases, \eqref{omegabound} holds.
\end{proof}

\begin{proposition} \label{prop:contuc}
Let $\sigma\in \SQRTO$
and $f\colon[0,+\infty)\to\mathbb{C}$ be the extension of $\sigma$
defined by \eqref{ext}.
Then $f\in \operatorname{RO}([0,+\infty))$.
\end{proposition}

\begin{proof}
The assumption $\sigma\in\SQRTO$ guarantees that
the right-hand side of \eqref{omegabound} tends to $0$ as $\delta$ tends to $0$.
\end{proof}

Note that Lemma~\ref{lem:contestimate}
and Proposition~\ref{prop:contuc} stay true
for every metric $\rho$ of the form
$\rho(x,y)=|\tau(x)-\tau(y)|$,
where $\tau\colon[0,+\infty)\to[0,+\infty)$
is a strictly increasing function satisfying
$\tau(n+1)-\tau(n)\le 1$ for every $n\in\Entero_+$.
In particular, applying this construction with
$\tau(n)=\ln(n+1)$
we obtain another proof
of \cite[Theorem 2.3]{HerreraMaximenkoVasilevski}
about the class $\operatorname{SO}(\Entero_+)$;
the proof in \cite{HerreraMaximenkoVasilevski}
is based on the usual piecewise-linear interpolation.

\section{Sqrt-oscillating property of the eigenvalues' sequences}\label{sqrt-property-of-gammas}

In this section we show that $\gamma_{a}\in\SQRTO$ for all $a\in\,L_{\infty}(\Real_{+})$.
From now on, 
we write the eigenvalues' sequence $\gamma_a$ as follows:
\begin{equation}\label{gamma-kernel-K}
\gamma_a(n) = \int_{\mathbb{R}_+} a(\sqrt{r})\,K(n,r)\,dr,
\quad\text{where}\quad 
K(n,r)=\frac{r^n e^{-r}}{n!},\quad n\in\Entero_{+}.
\end{equation}
The following proposition introduces a metric on $\mathbb{Z}_+$
which is, in a certain sense,
the most ``natural'' for the functions $\gamma_a$.

\begin{proposition}
Let $\kappa\colon\mathbb{Z}_{+}\times\mathbb{Z}_{+}\to [0,+\infty)$
be the function given by
\begin{equation}\label{eq:def_kappa}
\kappa(m,n)=\sup_{\stackrel{a\in\,L_{\infty}(\mathbb{R}_{+})}{\|a\|_{\infty}=1}}\left|\gamma_{a}(m)-\gamma_{a}(n)\right|.
\end{equation}
Then 
\begin{equation}\label{funcion:kappa}
\kappa(m,n)=\int_{\mathbb{R}_{+}}\left|K(m,r)-K(n,r)\right|dr.
\end{equation}
\end{proposition}

\begin{proof}
For every $a\in\,L_{\infty}(\mathbb{R}_{+})$ and $m,n\in\mathbb{Z}_{+}$ we have
\[
\left|\gamma_{a}(m)-\gamma_{a}(n)\right|
\leq \|a\|_{\infty}\int_{\mathbb{R}_{+}}\left|K(m,r)-K(n,r)\right|\,dr.
\]
On the other hand, if $m$ and $n$ are fixed and $m\neq\, n$,
we define $a_0\colon\mathbb{R}_{+}\to\mathbb{R}$ by $
a_{0}(r)=\operatorname{sign}\left(K(m,r)-K(n,r)\right)$, thus $a_{0}\in\,L_{\infty}(\mathbb{R}_{+})$ with  $\|a_{0}\|_\infty=1$, and
\[
\kappa(x,y)\geq \left|\gamma_{a_{0}}(m)-\gamma_{a_{0}}(n)\right|=\int_{\mathbb{R}_{+}}\left|K(m,r)-K(n,r)\right|\,dr.
\qedhere
\]
\end{proof}

\begin{lemma}
For every $n\in\mathbb{N}$ we get
\begin{equation}\label{kappa:n-1:n}
\kappa(n-1,n)=\frac{2n^{n}e^{-n}}{n!}.
\end{equation}
\end{lemma}

\begin{proof}
Given $n\in\mathbb{N}$, we write $\kappa(n-1,n)$
using \eqref{funcion:kappa}:
\[
\kappa(n-1,n)=\int_{0}^{+\infty}\left|\frac{r^{n-1}e^{-r}}{(n-1)!}-\frac{r^{n}e^{-r}}{n!}\right|dr=\int_{0}^{+\infty}\frac{r^{n-1}e^{-r}}{(n-1)!}\left|1-\frac{r}{n}\right|\,dr.
\]
Now the integral falls naturally into two parts:
\begin{align*}
\kappa(n-1,n)
&=\frac{1}{(n-1)!}\left[\int_{0}^{+\infty}e^{-r}\left(\frac{r^{n}}{n}-r^{n-1}\right)dr+2\int_{0}^{n}e^{-r}\left(r^{n-1}-\frac{r^{n}}{n}\right)dr\right]\\
&=\frac{2}{(n-1)!}\int_{0}^{n}e^{-r}\left(r^{n-1}-\frac{r^{n}}{n}\right)dr
\\
&=\frac{2}{(n-1)!}\left[\int_{0}^{n}e^{-r} r^{n-1}\,dr
-\int_{0}^{n}e^{-r}\,\frac{r^{n}}{n}\,dr\right].
\end{align*}
Integrating by parts in the latter integral
one gets \eqref{kappa:n-1:n}.
\end{proof}

\begin{lemma}\label{lema:behavior_kappa}
For each $n\in\mathbb{N}$ we have
\begin{equation}\label{eq:kappa_upper_bound}
\kappa(n-1,n)\leq\sqrt{\frac{2}{\pi\,n}}.
\end{equation}
Moreover,
\begin{equation}\label{eq:kappa_limit_relation}
\lim_{n\to\infty}\left(\kappa(n-1,n)\,\sqrt{n}\right)
=\sqrt{\frac{2}{\pi}}.
\end{equation}
\end{lemma}

\begin{proof}
The upper bound \eqref{eq:kappa_upper_bound}
follows from the left part of the well-known estimates
\begin{equation}\label{desigualdad:stirling}
n^{n}e^{-n}\sqrt{2n\pi}
\leq\,n!
\leq\,n^{n}e^{-n}\sqrt{2n\pi}\,e^{\frac{1}{12n}}.
\end{equation}
The limit relation \eqref{eq:kappa_limit_relation}
is a consequence of Stirling formula.
\end{proof}

%

\begin{proposition}
$\Gammas\subseteq\SQRTO$.
\end{proposition}

\begin{proof}
Let $a\in\,L_{\infty}(\Real_{+})$. Then for every $n\in\mathbb{Z}_+$
\[
|\gamma_{a}(n)|
\le \|a\|_\infty \int_{\mathbb{R}_+} K(n,r)\,dr
= \|a\|_\infty.
\]
Furthermore, by definition \eqref{eq:def_kappa} of $\kappa$
and Lemma~\ref{lema:behavior_kappa},
for every $n\in\N$
\[
\left|\sqrt{n}\left(\gamma_{a}(n)-\gamma_{a}(n-1)\right)\right|
\leq \|a\|_\infty\kappa(n,n-1)\sqrt{n}
\leq \sqrt{\frac{2}{\pi}}\,\|a\|_\infty.
\]
Thus $\gamma_{a}$ is Lipschitz continuous with respect to $\rho$
by Proposition~\ref{cond:Sup-Lip}.
\end{proof}

\begin{figure}[ht]
\includegraphics{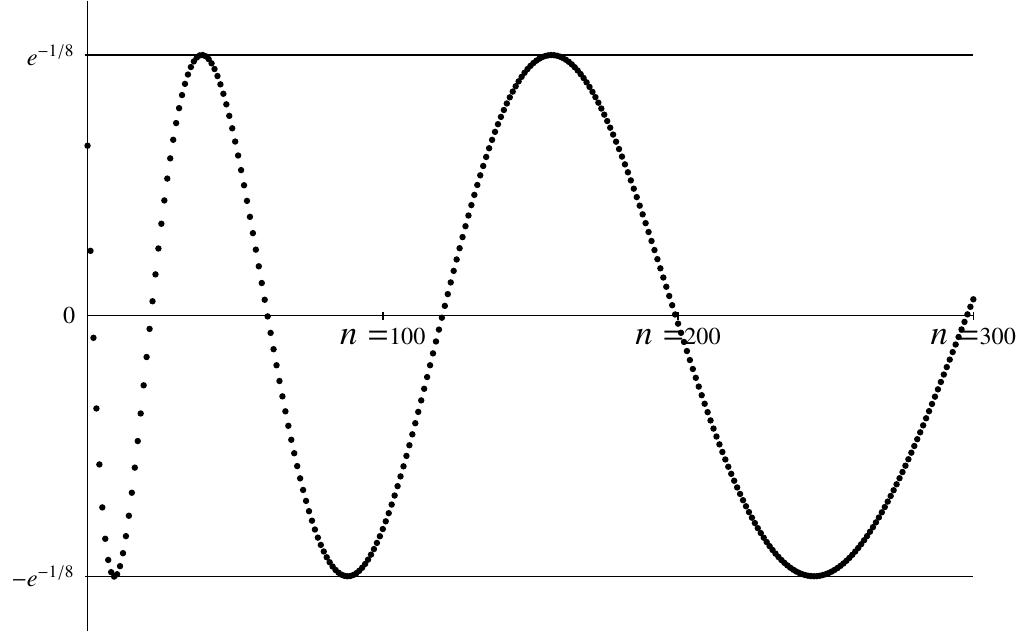}
\caption{The first $301$ values of the sequence $\gamma_a$
from Example~\ref{typicalexample}.}
\label{fig:typicalexample}
\end{figure}

\begin{example}[sqrt-oscillating eigenvalues' sequence]
\label{typicalexample}
Consider the Toeplitz operator generated
by the radial symbol $a(r)=\cos r$.
The corresponding eigenvalues are
\[
\gamma_a(n) = M(1+n, 1/2, -1/4),
\]
where $M={}_1F_1$ is the Kummer's
confluent hypergeometric function.
Using Proposition~\ref{convo-gamma} one can deduce
an asymptotic formula for $\gamma_a(n)$, as $n\to\infty$:
\[
\gamma_a(n) = e^{-1/8}\cos\sqrt{n} + o(1).
\]
Figure~\ref{fig:typicalexample} shows a plot
of $\gamma_a(n)$ for $n=0,1,\ldots,300$.
\end{example}

\section{Approximation of uniformly continuous functions by convolutions} \label{approx-unit-cont-func-by-conv}

In this section we recall a technique
that permits to approximate bounded uniformly continuous functions
by convolutions with a fixed kernel satisfying Wiener's condition.
We could not find Proposition~\ref{densidad:convolucion}
in the literature, but it is based on well-known ideas
and can be considered as a variation of the Wiener's Tauberian theorem.
The constructions of this section
can be generalized to abelian locally compact groups.
\begin{definition}[Dirac sequences]\label{def-dirac-suc}
A sequence $(h_{n})_{n\in\mathbb{N}}$
of functions belonging to $L_{1}(\Real)$
is called a \emph{Dirac sequence}
if it satisfies the following conditions:\\
$(\mathrm{a})$ For each $n\in\mathbb{N}$ and $x\in\Real$,
one gets $h_{n}(x)\geq0$.\\
$(\mathrm{b})$ For each $n\in\mathbb{N}$, 
\[
\int_{\Real}h_{n}(t)\,dt=1.
\]
$(\mathrm{c})$ For every $\delta>0$,
\[
\lim_{n\rightarrow+\infty}\int_{|x|>\delta}h_{n}(t)\,dt=0.
\]
\end{definition}

\begin{example}
The sequence $(h_{n})_{n\in\mathbb{N}}$ given by 
\begin{equation}\label{dirac-sin}
h_{n}(x)=\frac{2\sin^{2}(nx)}{\pi\,nx^{2}},\quad x\in\Real,
\end{equation}
is known as the \emph{Fej\'{e}r's kernel on $\Real$}
and is a Dirac sequence.
Moreover, it has a useful property that
all the functions $\hat{h}_n$ have compact supports.
\end{example}

Since the Dirac sequences can be viewed as approximate identities,
they provide a powerful tool to approximate functions.
The next lemma is a well-known result
for uniformly continuous functions,
see for example \cite[Proposition 2.42]{Folland}.

\begin{lemma}\label{Dirac:uniformly}
Let $f\in C_{b,u}(\Real)$.
If $(h_{n})_{n\in\mathbb{N}}$ is a Dirac sequence, then 
\begin{equation}\label{limite-dirac}
\lim_{n\rightarrow\infty}\left\|f\ast h_{n}-f\right\|_{\infty}=0.
\end{equation}
\end{lemma}


\begin{proposition}\label{densidad:convolucion}
Let $k\in L_{1}(\Real)$ satisfy Wiener's condition:
$\hat{k}(t)\neq0$ for each $t\in\Real$.
Then $\left\{k\ast f\colon f\in L_{\infty}(\Real)\right\}$
is a dense subset of $C_{b,u}(\Real)$.
\end{proposition}

\begin{proof}
It is known that $L_{1}(\Real)\ast\,L_{\infty}(\Real)=C_{b,u}(\Real)$ (see \cite[p.~283, 32--45]{Hewitt-Ross}),
hence every function in
$\left\{k\ast f\colon f\in L_{\infty}(\Real)\right\}$ belongs to $C_{b,u}(\Real)$.  Next, the density is proved by means of Wiener's Division Lemma  and Lemma \ref{Dirac:uniformly} as follows:
Let $(h_{n})_{n\in\mathbb{N}}$ be a Dirac sequence
such that the functions $\hat{h}_n$ have compact supports.
For example, $(h_{n})_{n\in\mathbb{N}}$
can be defined by \eqref{dirac-sin}.
Since $\widehat{k}(t)\neq0$ for each $t\in\Real$,
by Wiener's Division Lemma \cite[Lemma 1.4.2]{ReiterStegeman}
for every $n\in\mathbb{N}$
there exists $q_{n}\in\,L_{1}(\Real)$
such that $h_{n}=k\ast q_{n}$.
Now, given $\psi\in\,C_{b,u}(\Real)$,
we construct a sequence $(w_n)_{n\in\mathbb{N}}$ by the rule
$w_n=q_n\ast \psi$.
Then $w_{n}\in L_{\infty}$
and the sequence $(k\ast w_n)_{n\in\mathbb{N}}$
takes values in the set
$\{k\ast f\colon f\in L_{\infty}(\Real)\}$.
Finally, applying the identities
\[
k\ast w_{n}=k\ast q_n\ast \psi=h_{n}\ast \psi
\]
and Lemma~\ref{Dirac:uniformly},
we conclude that this sequence converges uniformly to $\psi$.
\end{proof}

\section{Approximation of the eigenvalues' sequences
by convolutions}\label{approx-convoluzation}

The idea of this section is to approximate
$\gamma_a(n)$ by a certain convolution for $n$ large enough.
Using the change of variables $r=y^2$
in \eqref{gamma-kernel-K}
we rewrite $\gamma_a(n)$ in the form
\begin{equation}\label{gamma-kernel-y}
\gamma_a(n)
=\int_{\Real_{+}}K(n,y^2)2y\,a(y)\,dy.
\end{equation}
By Stirling formula, $K(n,r)$ has the following
asymptotic behavior as $n\to+\infty$:
\[
K(n,r) = \frac{r^n e^{-r}}{n!}
\sim \frac{r^n e^n}{\sqrt{2\pi}\;n^{n+1/2} e^r}.
\]
Using this limit relation and
Lebesgue's dominated convergence theorem
it is easy to prove that
\begin{equation}\label{K-aproxi-F}
\lim_{n\to\infty}\int_{\Real_{+}}\left|\frac{r^n e^{-r}}{n!}- \frac{r^n e^n}{\sqrt{2\pi}\;n^{n+1/2} e^r}\right|\,dr=0.
\end{equation}
We pass from integer $n$ to real $x=\sqrt{n}$ and from $r\ge0$ to $r=y^{2}$. Consider the function $F$ defined on $[0,+\infty)\times [0,+\infty)$ by
\[
F(x,y) = \frac{y^{2x^2+1} e^{x^2}}{x^{2x^2+1} e^{y^2}}
= \exp\left((2x^2+1)(\ln y - \ln x) + x^2 - y^2\right).
\]
Then \eqref{K-aproxi-F} can be rewritten in the form
\begin{equation}\label{K-aproxi-F2}
\lim_{n\to\infty}\int_{\Real{+}}\left|K(n,y^2)\,2y
-\left(\frac{2}{\pi}\right)^{1/2} F(\sqrt{n},y)\right|\,dy=0.
\end{equation}
With the change of variables $u=y-x$ we have
\begin{equation}\label{eq:int_shifted}
\int_{\RPlus} F(x,y) a(y)\,dy = \int_{[-x,+\infty)} F(x,x+u) a(x+u)\,dy,
\end{equation}
where
\begin{equation}\label{eq:F}
\ln F(x,x+u) = (2x^2 + 1) \ln\left(1+\frac{u}{x}\right) - 2xu - u^2.
\end{equation}
Next, we proceed with some technical lemmas which permit us
to analyze the  asymptotic behavior
of the eigenvalues' sequences at the infinity.

\begin{lemma}[the integral of the kernel
far from the diagonal]\label{lemma:far_from_diagonal}
For every $\eps>0$ there exists $h>1$ such that
the following estimation holds for every $x\ge 1$:
\[
\int_{[-x,+\infty)\setminus[-h,h]} F(x,x+u) \,du \le \eps.
\]
\end{lemma}

\begin{proof}
Apply the elementary inequality $\ln(1+t)\le t$
which holds for every $t\ge 0$:
\begin{equation}\label{Eq:LNF}
\ln F(x,x+u) \le (2x^2 + 1)\,\frac{u}{x} - 2xu - u^2 = \frac{u}{x} - u^2\leq \frac{|u|}{x}-u^{2}.
\end{equation}
Suppose that $x\ge 1$, $h\ge 2$ and $|u|\ge h$. Since $|u|\geq\,h\geq2$, we have that $\frac{|u|}{2}\geq 1$. Thus by \eqref{Eq:LNF} we get
\[
\ln F(x,x+u) \leq \frac{|u|}{x}-u^{2}\leq |u|\left(\frac{|u|}{2}\right)-u^{2}= \frac{u^2}{2} - u^2 = -\frac{u^2}{2}.
\]
It follows that
\begin{align*}
\int_{[-x,+\infty)\setminus[-h,h]}\hspace{-0.25cm} F(x,x+u)\,du&\leq \int_{[-x,+\infty)\setminus[-h,h]}\hspace{-0.25cm}e^{-\frac{u^{2}}{2}}du\leq\int_{\Real\setminus[-h,h]}\hspace{-0.25cm}e^{-\frac{u^{2}}{2}}du\leq 2\int_h^{+\infty} e^{-\frac{u^2}{2}}\,du.
\end{align*}
The latter integral tends to zero as $h$ tends to $+\infty$.
\end{proof}

\begin{lemma}\label{lemma:ine-log}
Let $L, x\geq0$, $h\geq1$ and $|u|\leq h$. If $h<L<x$, then
\begin{equation}
 \left|F(x,x+u)e^{2u^{2}}-1\right|\leq\,e^{5\frac{h^{3}}{L}}-1. 
\end{equation}
\end{lemma}

\begin{proof}Since  $\ln(1+t)\leq\,t-\frac{t^{2}}{2}+\frac{t^{3}}{3},\quad t\in(-1,1)$, we obtain for  $t=\left|\frac{u}{x}\right|\leq 1$ that
\begin{align*}
\ln F(x,x+u) +2 u^2
&= (2x^2 + 1) \ln\left(1+\frac{u}{x}\right) - 2xu+u^{2}\\
& \le (2x^2 + 1)\left( \frac{u}{x}-\frac{u^{2}}{2x^{2}}+\frac{u^{3}}{3x^{3}}\right) - 2xu+u^{2}\\
&= \frac{u}{x}-\frac{u^{2}}{2x^{2}}+\frac{u^{3}}{3x^{3}}+\frac{2u^{3}}{3x}\leq \frac{u}{x}+\frac{u^{3}}{x}\leq\frac{5h^{3}}{L}.
\intertext{On the other hand, by $t-\frac{t^{2}}{2}\leq \ln(1+t)$ for each $t\in[0,1)$, taking $t=\frac{u}{x}$ with $u\in[0,h]$}
\ln F(x,x+u) +2 u^2
&= (2x^2 + 1) \ln\left(1+\frac{u}{x}\right) - 2xu+u^{2}\\
& \geq (2x^2 + 1)\left( \frac{u}{x}-\frac{u^{2}}{2x^{2}}\right) - 2xu+u^{2}= \frac{u}{x}-\frac{u^{2}}{2x^{2}}\geq-\frac{5h^{3}}{L}.
\end{align*}
Since $\ln(1-t)\geq -t-\frac{t^{2}}{2}-t^{3}$ for each $t\in[0,2/3]$, we take $x>0$ sufficiently large such that $t=-\frac{u}{x}\in[0,2/3]$, with $u\in[-h,0]$. Therefore  
\begin{align*}
\ln F(x,x+u) +2 u^2&= (2x^2 + 1) \ln\left(1-\left(-\frac{u}{x}\right)\right) - 2xu+u^{2}\\
&\geq(2x^{2}+1)\left(\frac{u}{x}-\frac{u^{2}}{2x^{2}}+\frac{u^{3}}{x^{3}}\right)-2xu+u^{2}\\
&=\frac{u}{x}-\frac{u^{2}}{2x^{2}}+\frac{u^{3}}{x^{3}}+\frac{2u^{3}}{x}\geq-\frac{5h^{3}}{L}.
\end{align*}
Combining this calculations we get for all $|u|\leq\,h$ that
\[ e^{5\frac{h^{3}}{L}}-1\geq F(x,x+u)e^{u^{2}}-1\geq\,e^{-5\frac{h^{3}}{L}}-1\geq-(e^{5\frac{h^{3}}{L}}-1).\qedhere\]
\end{proof}

\begin{lemma}[``convoluzation'' of the integral operator near the diagonal]\label{lemma:near_diagonal}
Given $\eps>0$ and $h\geq1$, there exists $L\ge h$ such that for every $x\ge L$
\[
\int_{[-h,h]} |F(x,x+u) - e^{-2u^2}|\,du \le \eps.
\]
\end{lemma}

\begin{proof}
Suppose that $x\ge L$ and $|u|\le h$. By Lemma \ref{lemma:ine-log} for $h\geq1$ we get
\begin{align*}
\int_{[-h,h]} \hspace{-0.1cm}|F(x,x+u) - e^{-2u^2}|\,du
\le \int_{[-h,h]} \hspace{-0.1cm}e^{-2u^2}\,|F(x,x+u) e^{2u^2} - 1|\,du
\le 2h\,(e^{5h^{3}/L}-1).
\end{align*}
The last expression tends to $0$ as $L$ tends to $+\infty$.
\end{proof}

\begin{lemma}\label{lem-aprox-nucleos}
\begin{equation}\label{eq:convoluzation}
\lim_{x\to+\infty}\int_{0}^{\infty}
\left|F(x,y)-e^{-2(x-y)^{2}}\right|dy=0,
\end{equation}
\begin{equation}\label{eq:convoluzationK}
\lim_{n\to\infty}\int_{\Real_{+}}
\left|K(n,y^2)\,2y-\left(\frac{2}{\pi}\right)^{1/2}
e^{-2(\sqrt{n}-y)^2}\right|\,dr=0.
\end{equation}
\end{lemma}

\begin{proof}
We are going to prove \eqref{eq:convoluzation},
then \eqref{eq:convoluzationK} will follow by \eqref{K-aproxi-F}.
Let $\varepsilon>0$.
Using Lemma~\ref{lemma:far_from_diagonal} choose $h>0$ such that
\[
\int_{[-x,+\infty)\setminus[-h,h]} F(x,x+u) \,du \le \frac{\eps}{3\|a\|_\infty},\quad
\int_{[-x,+\infty)\setminus[-h,h]} e^{-2u^2}\,du\le\frac{\eps}{3\,\|a\|_\infty}.
\]
After that using Lemma~\ref{lemma:near_diagonal}
choose $L\ge h$ such that for every $x\ge L$
\[
\int_{[-h,h]} |F(x,x+u) - e^{-2u^2}|\,du \le \frac{\eps}{3\,\|a\|_\infty}.
\]
Then for every $x\ge L$ the left-hand side of \eqref{eq:convoluzation} is less or equal to $\eps$.
\end{proof}
%

The proofs of this section have many technical details.
To be more confident in formula \eqref{eq:convoluzationK},
we tested it numerically in Wolfram Mathematica.
The numerical experiments showed
that for every $n\in\{1,\ldots,1000\}$
the integral in the left-hand side
of \eqref{eq:convoluzationK}
is less than $0.54/\sqrt{n}$.

\begin{proposition}\label{convo-gamma}
Let $a\in L_\infty(\RPlus)$. Then
\begin{equation}\label{behavior-gamma}
\lim_{n\to+\infty}
\left|\gamma_{a}(n)-\left(\frac{2}{\pi}\right)^{1/2}
\int_{\RPlus}a(y)\,e^{-2(y-\sqrt{n}\,)^{2}}\,dy\right|=0.
\end{equation}
\end{proposition}
\begin{proof}
Write $\gamma_a(n)$ as in \eqref{gamma-kernel-y},
factorize $a(y)$ below the sign of the the integral,
estimate $|a(y)|$ by $\|a\|_\infty$
and apply \eqref{eq:convoluzationK}.
\end{proof}

There is no surprise that the heat kernel appears
in the properties of the eigenvalues' sequences $\gamma_{a}$,
because it plays an important role
in the theory of Toeplitz operators acting on Fock spaces.
In \cite{Berger-Coburn-2} Berger and Coburn
characterized some properties of Toeplitz operators $T_{\varphi}$
(boundedness, compactness etc.) by means of its Berezin transform
\[
\widetilde{\varphi}(z)=\frac{1}{\pi}\int_{\Real}\varphi(w)e^{-\frac{|z-w|^{2}}{2}}d\nu(w),\quad z\in\Complex,
\]
which is the convolution of the symbol $\varphi$ with the heat kernel $H(w,t)=(4t\pi)^{-1}e^{-\frac{|w|^{2}}{4t}}$
at time $t=\frac{1}{2}$. This result holds also for Toeplitz operators with more general symbols (positive Borel measures), see \cite{Isra-Zhu}.

Our formula \eqref{behavior-gamma} 
relates $\gamma_a$ with the heat kernel at time $t=\frac{1}{8}$,
we denote it simply by $H$:
\[
H(x)=H(x,1/8)=(2/\pi)^{1/2}e^{-2x^{2}}.
\]

\begin{lemma}\label{aprox:+infinity}
If $b\in\,L_{\infty}(\Real)$ and $a=\chi_{\Real_{+}}b$, then
\begin{equation}\label{eq:aprox:+infinity}
\lim_{x\rightarrow+\infty}\left|H\ast a(x)-H\ast b(x)\right|=0.
\end{equation}
\end{lemma}
\begin{proof}
The difference in the left-hand side of
\eqref{eq:aprox:+infinity} can be estimated as follows:
\begin{align*}
\left|H\ast a(x)-H\ast b(x)\right|&\leq \|b\|_{\infty}(2/\pi)^{1/2}\int_{-\infty}^{0}e^{-2(x-y)^{2}}dy\\
&\stackrel{t=x-y}{=} \|b\|_{\infty}(2/\pi)^{1/2}\int_{x}^{+\infty}e^{-2t^{2}}dt,\quad x\in\Real_{+}.\quad\qedhere
\end{align*}
\end{proof}


\begin{proposition}\label{prop:cola}
Let $\sigma\in\SQRTO$ and $\varepsilon>0$.
Then there exist
$a\in\,L_{\infty}(\Real_{+})$ and $N\in\N$ such that 
\begin{equation}
\sup_{n> N}\left|\sigma(n)-\gamma_{a}(n)\right|\leq\varepsilon.
\end{equation}
\end{proposition}
\begin{proof}
By Proposition~\ref{prop:contuc}
there is $f\in\operatorname{RO}([0,+\infty))$ such that $f|_{\Entero_{+}}=\sigma$
and $\|f\|_{\infty}=\|\sigma\|_{\infty}$.
Define $h\colon\Real\to\Complex$ as $h(x)=f(x^{2})$.
Then $h\in C_{b,u}(\Real)$.
Moreover, by Proposition~\ref{densidad:convolucion}
there exists $\ell\in\,L_{\infty}(\Real)$ such that 
\begin{equation}\label{1}
\|H\ast \ell-h\|_{\infty}\leq\frac{\varepsilon}{3}.
\end{equation}
Denote the restriction $\ell|_{\Real_{+}}$ by $a$.
By \eqref{behavior-gamma} and Lemma~\ref{aprox:+infinity},
there are $L_{1},L_{2}>0$ such that 
\begin{align}\label{2}
&\left|\gamma_{a}(n)-H\ast a (\sqrt{n}\,)\right|\leq\frac{\varepsilon}{3},\, n\geq L_{1},\quad\left|H\ast \ell(x)-H\ast a (x)\right|\leq\frac{\varepsilon}{3},\, x\geq L_{2}.
\end{align}
Thus, taking  $L=\max\{L_{1},L_{2}\}$ by \eqref{1} and \eqref{2} one gets for every $n\geq \lceil L^{2}\rceil$ that
\begin{align*}
\left|\gamma_{a}(n)-\sigma(n)\right|&\stackrel{x=\sqrt{n}}{\leq}\left|\gamma_{a}(x^{2})-H\ast a(x)\right|+\left|H\ast a(x)-H\ast \ell(x)\right|+\left|H\ast \ell(x)-h(x)\right|\\
&\leq\left|\gamma_{a}(x^{2})-H\ast a(x)\right|+\left|H\ast a(x)-H\ast \ell(x)\right|+\left\|H\ast \ell-h\right\|_{\infty}\leq\varepsilon.\qedhere
\end{align*}
\end{proof}
%
\section{Density of $\mathfrak{G}$ in $\SQRTO$}\label{density-in-c0}
%

In this section we finish the proof of our main result.
By Proposition~\ref{prop:cola} we already know
that every sequence $\sigma\in\SQRTO$ can be approximated
by some eigenvalues' sequences $\gamma_{a}$
for large values of $n$.
Thus, it only remains to prove that the sequences vanishing at the infinity
can be approximated by eigenvalues' sequences.



Denote by $\X$ the Banach subspace of $L_{\infty}(\Real_{+})$ consisting of all bounded functions $a$  having limit $0$ at the infinity.

\begin{lemma}\label{gamma-in-c0}
If $a\in\X$, then $\gamma_{a}\in\,c_{0}(\Entero_{+})$.
\end{lemma}
%
%
\begin{proof}
Given $\varepsilon>0$,
there are $L>0$ and $N_{0}\in\Entero_{+}$ such that
\begin{equation}\label{eq:symbol-sqrtn}
|a(t)|\leq\frac{\varepsilon}{2},\quad t\geq L,\quad n^{-1/2}\leq \sqrt{\frac{\pi}{2}}\frac{\varepsilon}{\|a\|_{\infty}L^{2}},\quad n\geq\,N_{0}.
\end{equation}
Thus, by \eqref{desigualdad:stirling} and \eqref{eq:symbol-sqrtn} we have for every $ n\geq N_{0}$ that
\begin{align*}
|\gamma_{a}(n)|&\leq \frac{1}{n!}\left[\int_{0}^{L^{2}}|a(\sqrt{r})|\,e^{-r}r^{n}dr+\int_{L^{2}}^{+\infty}|a(\sqrt{r})|\,e^{-r}r^{n}dr\right]\\
&\leq\frac{1}{n!}\left[\int_{0}^{L^{2}}|a(\sqrt{r})|\,e^{-r}r^{n}dr+\frac{\varepsilon}{2}\,\int_{L^{2}}^{+\infty}e^{-r}r^{n}dr\right]\\
&\leq \frac{1}{n!}\int_{0}^{L^{2}}|a(\sqrt{r})|\,e^{-r}r^{n}dr+\frac{\varepsilon}{2}\leq \frac{\|a\|_{\infty}}{n!}\int_{0}^{L^{2}}\,e^{-r}r^{n}dr+\frac{\varepsilon}{2}\\
&\leq \frac{\|a\|_{\infty}e^{-n}n^{n}L^{2}}{n!}+\frac{\varepsilon}{2}\leq \frac{\|a\|_{\infty}L^{2}}{\sqrt{2n\pi}}+\frac{\varepsilon}{2}=\frac{\varepsilon}{2}+\frac{\varepsilon}{2}=\varepsilon.\qedhere
\end{align*}
\end{proof}

%
%
\begin{theorem}\label{teor:densidad:c0}   $\left\{\gamma_a\colon\ a\in\X\right\}$ is a dense subset of $c_0(\Entero_{+})$.
\end{theorem}
%

\begin{proof}
The inclusion
$\left\{\gamma_a\colon\ a\in\X\right\}\subseteq c_0(\Entero_{+})$ was shown in Lemma~\ref{gamma-in-c0}. Unfortunately we were not able
to prove the density by constructive tools;
the next proof uses non-constructive duality arguments. By Hahn--Banach theorem,
the density of $\{\gamma_a\colon\ a\in \X\}$
in $c_{0}(\Entero_{+})$ will be shown
if we prove that any continuous linear functional
$\varphi$ on $c_{0}(\Entero_{+})$
that vanishes on $\{\gamma_a\colon\ a\in \X\}$
is the zero functional.
Thus, let $\phi\in c_{0}(\Entero_{+})^\ast$
be a linear functional such that $\phi(\gamma_{a})=0$ for each $a\in\,L_{\infty}(\Real_{+})$.
Using the well-known description of the dual space of
$c_{0}(\Entero_{+})$ we find a sequence
$p=(p_{n})_{n\in\Entero_{+}}\in\ell_{1}(\Entero_{+})$ such that
\[
\phi(y)=\sum_{n=0}^{\infty}p_{n}y_{n}
\qquad y\in c_{0}(\Entero_{+}).
\]
Then we have that
\[
0=\phi(\gamma_{a})
=\sum_{n=0}^{\infty}\gamma_{a}(n)p_{n},\quad a\in\,L_{\infty}(\Real_{+}).
\]
In particular, substituting
$a=\chi_{[0,x]}\in\X$ with $0\leq\,x<+\infty$, we obtain
\[
0=\sum_{n=0}^{\infty}\gamma_{a}(n)p_{n}
=\int_{0}^{x}\sum_{n=0}^{\infty}p_{n}K(n,r)\,dr.
\]
The function $r\mapsto\sum_{n=0}^{\infty}p_{n}K(n,r)$,
being the sum of a uniformly converging series of continuous functions,
is continuous, and by the first fundamental theorem of calculus, $
\sum_{n=0}^\infty p_n K(n,r) = 0,\, r\ge0.$
Now, replace  $K(n,r)$ by $r^n e^{-r}/n!$ and factorize $e^{-r}$:
\begin{equation}\label{eq:3a}
\sum_{n=0}^\infty \frac{p_n r^n}{n!} = 0\qquad r\ge0.
\end{equation}
Denote by $f$ the function
\[
f(z)=\sum_{n=0}^\infty \frac{p_n\,z^n}{n!}.
\]
Since $p_n\to0$, the root test shows that $f$ is an entire function.
The equality \eqref{eq:3a} says that $f(r)=0$ for every $r\ge0$.
Therefore $f$ is the zero constant,
and all coefficients $p_n$ are zero.
\end{proof}

Now we are ready to prove the main result of the paper.

\begin{theorem}
$\mathfrak{G}$ is dense in $\SQRTO$.
\end{theorem}

\begin{proof}
Let $\sigma\in\SQRTO$ and $\varepsilon>0$.
By Proposition~\ref{prop:cola}
there are $b\in\,L_{\infty}(\Real_{+})$
and $N\in\Entero_{+}$ such that 
$$|\sigma(n)-\gamma_{b}(n)|\leq\frac{\varepsilon}{2},\quad n> N.$$
Define  $\vartheta=(\vartheta(n))_{n\in\Entero_{+}}$ by
$$\vartheta(n)=\begin{cases} \sigma(n)-\gamma_{b}(n),&\hbox{if $n\leq\,N$,}\\ 0&\hbox{otherwise.}
\end{cases}$$
Thus $\vartheta\in\,c_{0}(\Entero_{+})$, and by Theorem \ref{teor:densidad:c0} there exists $c\in\,L_{\infty}(\Real_{+})$ such that
$$\|\vartheta-\gamma_{c}\|_{\infty}\leq\frac{\varepsilon}{2}.$$
Taking $a=b+c\in\,L_{\infty}(\Real_{+})$ one gets that
\[
\|\sigma-\gamma_{a}\|_{\infty}
\leq 
\|\sigma-\gamma_{b}-\vartheta\|_{\infty}
+ \|\vartheta-\gamma_{c}\|_{\infty}
\leq \sup_{n> N}|\sigma(n)-\gamma_{b}(n)|
+ \frac{\varepsilon}{2}
\leq\varepsilon.
\qedhere
\]
\end{proof}

\section{Beyond the class of bounded generating symbols}
\label{Unbounded-symbols}

%

In this section we indicate a class of functions wider than $L_\infty(\Real)$,
with eigenvalues' sequences  belonging to $\SQRTO$. Furthermore, we give an unbounded generating symbol $a$ such that $\gamma_a\in\ell_\infty(\Entero_+)\setminus\SQRTO$. 

Following \cite{Grudsky-Vasilevski} we denote by $L_{1}^{\infty}(\Real_{+},e^{-r^{2}})$
the subspace of all measurable functions $a$ on $\Real_{+}$
for which the following integrals are finite
for all $n\in\Entero_{+}$:
\begin{equation}
\int_{\Real_{+}}|a(r)|\,e^{-r^{2}}r^{n}dr<+\infty.
\end{equation}
For $a\in\,L_{1}^{\infty}(\Real_{+},e^{-r^{2}})$, we consider the following averages \cite{Grudsky-Vasilevski}:
\begin{equation}\label{promedio-Bj}
\B_{(j)}a(r)=\int_{r}^{+\infty}\B_{(j-1)}a(u)e^{r-u}du,\quad j=1,2,\ldots,
\end{equation}
where $\B_{(0)}a(r)=a(\sqrt{r})$.
Integrating by parts $j$ times
one can express $\gamma_a$ through $\B_{(j)}a$:
\begin{equation}\label{relacion-gamma-a-Bj}
\gamma_{a}(n)=\frac{1}{(n-j)!}\int_{\Real_{+}}\B_{(j)}a(r)r^{n-j}e^{-r}dr=\gamma_{\B_{(j)a}\circ\,\operatorname{sq}}(n-j),\quad n\geq j,
\end{equation}
where $\operatorname{sq}(x)=x^{2},\,x\in\Real_{+}.$
It is easily  seen that if $\B_{(j)}a\in\,L_{\infty}(\Real_{+})$ for some $j\in\Entero_{+}$, then the eigenvalues' sequence
(and the corresponding Toeplitz operator) is bounded.
The definition of the averages
$\B_{(j)}a$ and the facts mentioned above
are taken from \cite[Section 4]{Grudsky-Vasilevski}.
%
%

Let us denote by $\mathcal{M}$ the class of symbols $a\in\,L_{1}^{\infty}(\Real_{+},e^{-r^{2}})$ such that the average \eqref{promedio-Bj} is bounded for some $j\in\Entero_{+}$:
\begin{equation}
\mathcal{M}:=\left\{a\in\,L_{1}^{\infty}(\Real_{+},e^{-r^{2}})\colon \B_{(j)}a\in\,L_{\infty}(\Real_{+})
\ \text{for some}\ j\in\Entero_{+}\right\}.
\end{equation}

\begin{proposition}
If $a\in\mathcal{M}$, then $\gamma_{a}\in\SQRTO$.
\end{proposition}

\begin{proof}
Let $j\in\Entero_{+}$
and $a\in\,L_{1}^{\infty}(\Real_{+},e^{-r^{2}})$
such that $\B_{(j)}a\in\,L_{\infty}(\Real_{+})$. Since $\gamma_{a}(n)=\gamma_{\B_{(j)a}\circ\,\operatorname{sq}}(n-j)$, with $\operatorname{sq}(x)=x^{2},\,x\in\Real_{+}$, one gets by \eqref{eq:def_kappa} and \eqref{eq:kappa_upper_bound} that for every $n> j$ 
\begin{align*}
\sqrt{n+1}\,\left|\gamma_{a}(n)-\gamma_{a}(n+1)\right|&=\sqrt{n+1}\,\left|\gamma_{\B_{(j)}a\circ\,\operatorname{sq}}(n-j+1)-\gamma_{\B_{(j)}a\circ\,\operatorname{sq}}(n-j)\right|\\
&\leq \|\B_{(j)}a\circ\,\operatorname{sq}\|_{\infty}\kappa(n-j,n-j+1)\sqrt{n+1}\\
&\leq\|\B_{(j)}a\circ\,\operatorname{sq}\|_{\infty}\sqrt{\frac{2(n+1)}{\pi(n-j+1)}}.
\end{align*}

Hence
$\sup_{n\in\Entero_{+}}\left(\sqrt{n+1}\,\left|\gamma_{a}(n)-\gamma_{a}(n+1)\right|\right)<\infty$
and, by Proposition \ref{cond:Sup-Lip},
the eigenvalue sequence $\gamma_{a}$
is Lipschitz continuous with respect $\rho$.
\end{proof}

Folland  \cite[Lemma 2.95]{Folland2}  proved that for the class of unbounded measurable symbols $a\in\,L_{1}^{\infty}(\Real_{+},e^{-r^{2}})$ which satisfy the inequality
\begin{equation}\label{exponencialmente-acotados}
|a(r)|\leq const\, e^{\delta r^{2}},\quad\text{ for some}\quad\delta<1,
\end{equation}
the linear mapping $a\mapsto T_{a}$ is injective.
However,  this class contains  defining symbols which generate eigenvalues' sequences  do not belonging to $\SQRTO$. 


\begin{example}
Let $\delta=1-\frac{1}{\sqrt{2}}$.
Then the function $a$ defined by the rule
\[
a(r)=e^{\left(\delta-\frac{i}{\sqrt{2}}\right)r^{2}}
\]
satisfies \eqref{exponencialmente-acotados}
and  belongs to $L_{1}^{\infty}(\Real_{+},e^{-r^{2}})$.
Let us calculate the corresponding eigenvalue's sequence
using the change of variables $t=\sqrt{2}\,r$
and the formula \cite[Eq. 3.381-5]{Gradshteyn}:
\[
\gamma_{a}(n)=\frac{1}{n!}\int_{\Real_{+}}
e^{\left(1-\frac{1}{\sqrt{2}}-\frac{i}{\sqrt{2}}\right)\,r}
e^{-r}r^{n}\,dr
=\frac{2^{\frac{n+1}{2}}}{n!}\int_{\Real_{+}}e^{-(1+i)t}t^{n}\,dt
=e^{-i(n+1)\frac{\pi}{4}}.
\]
The sequence of its consecutive differences is given by
\begin{align*}
\gamma_{a}(n+1)-\gamma_{a}(n)=e^{-i(n+2)\frac{\pi}{4}}\left(1-e^{i\frac{\pi}{4}}\right)
\end{align*}
and does not converge to $0$, though $\rho(n+1,n)\to0$.
Thus $\gamma_a\in\ell_\infty(\Entero_+)\setminus\SQRTO$.
\end{example}
\subsection*{Acknowledgments}

The authors are grateful to Professor Nikolai Vasilevski
for introducing to us the world of commutative C*-algebras
of Toeplitz operators. Many ideas used in the proofs
come from our joint papers with Crispin Herrera Ya\~{n}ez,
Ondrej Hutn\'{i}k, and Nikolai Vasilevski.
The second author was partially supported
by the project IPN-SIP 2015-0422.

\end{document}